\documentclass[11pt]{amsart}
\usepackage{amsmath, amssymb,graphicx,mathrsfs,enumerate}
\usepackage[all]{xy}
\usepackage{hyperref}
\usepackage{pgf,tikz}
\usetikzlibrary{patterns,snakes}

\usepackage{caption}
\usepackage{subcaption}

\newtheorem{thm}{Theorem}
\newtheorem{thmA}{Theorem}

\newtheorem{prop}[thm]{Proposition}
\newtheorem{lem}[thm]{Lemma}

\newcommand{\rk}{\operatorname{rk}}
\newcommand{\crk}{\operatorname{crk}}

\newcommand{\cn}{\operatorname{cn}}

\begin{document}

\title[Covering numbers of unipotent conjugacy classes]{Covering numbers of unipotent conjugacy classes in simple algebraic groups}
\author{Iulian I. Simion}
\address[Iulian I. Simion]
        { Department of Mathematics\\
          Babeș-Bolyai University\\
          Str. Ploieşti 23-25, Cluj-Napoca 400157, Romania\\
          and
          Department of Mathematics\\
          Technical University of Cluj-Napoca\\
          Str. G. Bari\c tiu 25, Cluj-Napoca 400027, Romania}
\email{iulian.simion@ubbcluj.ro}
\thanks{I am grateful to Prof. Attila Mar\'oti for many discussions on this topic and to the referees for their remarks and suggestions. This work was supported by a grant of the Ministry of Research, Innovation and Digitalization, CNCS/CCCDI–UEFISCDI, project number PN-III-P1-1.1-TE-2019-0136, within PNCDI III}

\maketitle
\begin{abstract}
  For simple algebraic groups defined over algebraically closed fields of good characteristic, we give upper bounds on the covering numbers of unipotent conjugacy classes in terms of their (co)ranks and in terms of their dimensions.
\end{abstract}


\section{Introduction}

The covering number $\cn(G,S)$ of a subset $S$ of a group $G$ is the smallest integer $k$ such that $S^k=G$ or $\infty$ if no such $k$ exists. By a theorem of Liebeck and Shalev \cite[Theorem 1.1]{Liebeck_Shalev}, there is a constant $c$ such that whenever $C$ is a non-central conjugacy class of a non-abelian finite simple group $G$ we have $\cn(G,C)\leq c \cdot (\log_{2} |G|/ \log_{2} |C|)$.

Let $G$ be a simple algebraic group and let $C$ be a non-central conjugacy class of $G$. By results of Gordeev \cite{Gordeev_1_2}, if $G$ is defined over an algebraically closed field of characteristic $0$ then $\cn(G,C)\leq 4\cdot\rk(G)$ where $\rk(G)$ is the Lie rank of $G$. This result was extended by Ellers, Gordeev and Herzog \cite{Ellers_Gordeev_Herzog} to the case of quasisimple Chevalley groups. More precisely, they show that for such a group $G$ we have $\cn(G,C)\leq 2^{13}\cdot \rk(G)$. The generic upper bound for the covering number of a conjugacy class is linear in the Lie rank of $G$. Gordeev and Saxl \cite{Gordeev_Saxl_1} show that a similar upper bound holds for the extended covering number. In particular for a Chevalley group $G$ defined over an algebraically closed field they obtain $\cn(G,C)\leq 4\cdot \rk(G)$.

Throughout this paper $G$ denotes a simple algebraic group defined over an algebraically closed field of characteristic $p$. We are interested in upper bounds on $\cn(G,C)$ which take into account `the size' of $C$ as in \cite[Theorem 1.1]{Liebeck_Shalev}. The question of finding such a bound - with an explicit constant - should involve a classification of the conjugacy classes in $G$. We assume throughout that $p$ is a good prime for $G$, i.e. $p\neq 2$ if $G$ is not of type $A$, $p\neq 3$ if $G$ is an exceptional group and $p\neq 5$ if $G$ is of type $E_8$. We impose this restriction on $p$ in order to make use of the Bala-Carter-Pommerening classification of unipotent conjugacy classes \cite{Bala_Carter,Pommerening} (see also \cite[Theorem 5.9.6 and \S5.11]{Carter_Finite}). This classification reduces the study of unipotent conjugacy classes to the study of distinguished conjugacy classes. Recall that a unipotent element is distinguished if $C_G(u)^{\circ}$ is unipotent. Our first result gives an upper bound on the covering number of distinguished conjugacy classes.

\begin{thmA}
  \label{thm_distinguished}
  There is a constant $c$ such that for any simple algebraic group $G$ defined over a field of good characteristic and any distinguished unipotent conjugacy class $C$ of $G$ we have
  $$
  \cn(G,C)\leq c.
  $$
  Moreover, we may choose $c= 2^3\cdot 3^2$.
\end{thmA}

For bounded rank, in particular for exceptional algebraic groups, the above result and Theorems \ref{thm_A} and \ref{thm_B} below follow from \cite{Gordeev_Saxl_1}. It is worth noticing that if $G$ is adjoint of type $A$ and if $C$ is the regular unipotent conjugacy class then $C^2=G$ by \cite{Lev}. In other words, for such groups the constant in Theorem \ref{thm_distinguished} is $2$.


The rank $\rk(H)$ of an algebraic group $H$ is the dimension of a maximal torus of $H$. Let $C$ be the conjugacy class of the unipotent element $u\in G$. We define the \emph{corank of $C$} to be $\crk(C):=\rk(C_G(u))$. Further, we define the \emph{rank of $C$} to be $\rk(C):=\rk(G)-\crk(C)$. The second result gives an upper bound for $\cn(G,C)$ in terms of the rank and the corank of $C$.

\begin{thmA}
  \label{thm_A}
  There is a constant $c$ such that for any simple algebraic group $G$ defined over a field of good characteristic and any unipotent conjugacy class $C$ of $G$ we have
  $$
  \cn(G,C)\leq c\cdot \frac{\rk(G)}{\rk(C)}=
  c\cdot \left(1+\frac{\crk(C)}{\rk{(C)}}\right).
  $$
  Moreover, we may choose $c=2^5\cdot 3^2$.
\end{thmA}

The Lang-Weil bound \cite[Theorem 1]{Lang_Weil} suggests that the analogue of \cite[Theorem 1.1]{Liebeck_Shalev} for algebraic groups is $\cn(G,C)\leq c \cdot (\dim(G)/\dim(C))$ where $c$ is a universal constant independent of $G$. We prove this bound in the case of unipotent conjugacy classes.

\begin{thmA}
  \label{thm_B}
  There is a constant $c$ such that for any simple algebraic group $G$ defined over a field of good characteristic and any unipotent conjugacy class $C$ of $G$ we have
  $$
  \cn(G,C)\leq c\cdot \frac{\dim(G)}{\dim(C)}.
  $$
  Moreover, we may choose $c= 2^{9}\cdot 3^2$.
\end{thmA}

The focus of this paper is on unipotent conjugacy classes and the asymptotic bound of their covering numbers. We believe that it should be possible to improve the upper bound on the constant $c$.

The paper is structured as follows: Section \ref{preliminaries} introduces the notation needed throughout the paper, slightly extends the context to normal subsets, slightly extends the notion of covering number and introduces marked diagrams. 
The proofs of Theorems \ref{thm_distinguished}, \ref{thm_A} and \ref{thm_B} are given in Section \ref{sec_distinguished}, \ref{sec_A} and \ref{sec_B} respectively.

\section{Preliminaries}
\label{preliminaries}

\subsection{Setup}
In this paper $G$ denotes a simple algebraic group of rank $r=\rk(G)$ defined over an algebraically closed field $F$ of good characteristic $p$. We assume throughout that $p$ is a good prime for $G$. We fix a Borel subgroup $B$ with unipotent radical $U$ and maximal torus $T$. The roots $\Phi$ of $G$ are with respect to $T$, the set of positive roots $\Phi^{+}$ are with respect to $U$ and $\Delta$ denotes the set of simple roots of $\Phi$ in $\Phi^{+}$. We denote by $U^{-}$ the radical of the Borel subgroup opposite to $B$, i.e. $U^{-}=U^{\dot w_0}$ for some representative $\dot w_0\in N_G(T)$ of the longest element (with respect to $\Delta$) of the Weyl group $N_G(T)/T$. For an element $w\in N_G(T)/T$ we write $\dot w$ for a representative in $N_G(T)$.

For each root $\alpha\in\Phi$ let $u_{\alpha}:F\rightarrow U_{\alpha}$ be an isomorphism from the additive group of the ground field $F$ onto the root subgroup $U_{\alpha}$. For each $\alpha\in\Phi$ we denote by $\alpha^{\vee}:F^{\times}\rightarrow T$ the cocharacter corresponding to the root alpha (see \cite[II\S1.3]{Jantzen_Reductive}). Then
\begin{equation}
  \label{cochar_action}
        {}^{\alpha^{\vee}(t)}u_{\beta}(x)
        =\alpha^{\vee}(t)u_{\beta}(x)\alpha^{\vee}(t)^{-1}
        =u_{\beta}(\beta(\alpha^{\vee}(t))x)
        =u_{\beta}(t^{\langle\beta,\alpha\rangle}x)
\end{equation}
for all $\alpha,\beta\in\Phi$, $t\in F^{\times}$, $x\in F$ (see \cite[Ch.7]{Carter_Simple}).

For a set of roots $I\subseteq\Phi$, let $\Phi_{I}$ be the root subsystem $\langle I\rangle_{\Phi}$ generated by $I$. We denote by $L_I$ the subgroup $\langle T,U_{\alpha}:\alpha\in\Phi_{I}\rangle$ of $G$. If the roots in $I$ are simple then $L_I$ is a standard Levi subgroup. In this case, we denote by $P_I$ the standard parabolic subgroup with Levi factor $L_I$. When we need to specify the ambient group $G$, we write $L_I^{G}$ or $P_I^{G}$. Notice that $L_I^{G}$ and $P_I^{G}$ make sense in the more general case of a reductive algebraic group $G$. Moreover, we denote by $G(I)$ the subsystem subgroup $\langle U_{\alpha}:\alpha\in\Phi_{I}\rangle$. Notice that $L_I=G(I)T$. In the particular case of $I=\{\alpha\}\subseteq\Delta$ we denote by $G_{\alpha}$ the subgroup $G(I)$.


\subsection{Normal subsets} A product of conjugacy classes of $G$ is invariant under conjugation by $G$. A \emph{normal subset} $N$ of a group $G$ is a non-empty subset of $G$ which is invariant under $G$-conjugation, i.e. $N$ is a non-empty union of conjugacy classes of $G$. The intermediate steps in our proofs are easier to formulate with this notion. Notice that Theorems \ref{thm_distinguished}, \ref{thm_A} and \ref{thm_B} can be formulated for normal subsets consisting of unipotent elements if the conditions of those statements are imposed on the highest dimensional classes in $N$.

\subsection{Covering numbers} When dealing with successive powers of a conjugacy class $C$ of $G$, we use intermediate steps in which we show that for certain $k\geq 0$ the normal subset $C^k$ contains a certain subset of $G$. For this we slightly extend the notion of covering numer as follows. For subsets $S_1$ and $S_2$ of $G$ we define the \emph{$S_1$-covering number of $S_2$} to be the smallest integer $k$ such that $S_1\subseteq S_2^k$ or $\infty$ if no such $k$ exists. We denote this number by $\cn(S_1,S_2)$.

\subsection{Marked diagrams}
\label{subsection:marked_diagrams}
Let $I$ be a subset of $\Delta$. The \emph{marked diagram} $D_I$ corresponding to the set $I$ is the Dynkin diagram of $G$ where we mark the nodes corresponding to the roots in $\Delta-I$. For example, if $G$ is of type $A_4$, the marked diagram corresponding to $I=\{\alpha_2,\alpha_4\}$ is
$$
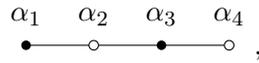
\begin{tikzpicture}[scale=.9]
  \node[label=$\alpha_{1}$] (N0) at (0,0) {};
  \node[label=$\alpha_{2}$] (N1) at (1,0) {};
  \node[label=$\alpha_{3}$] (N2) at (2,0) {};
  \node[label=$\alpha_{4}$] (N3) at (3,0) {};
  \draw (N0.center) -- (N1.center) -- (N2.center) -- (N3.center);
  \fill (N0) circle (2pt);
  \fill[white] (N1) circle (2pt);
  \draw (N1) circle (2pt);
  \fill (N2) circle (2pt);
  \fill[white] (N3) circle (2pt);
  \draw (N3) circle (2pt);
\end{tikzpicture},
$$
i.e. the marked nodes are the black nodes in the figure. We identify the roots in $\Delta$ with the corresponding nodes in the Dynkin diagram. A \emph{component} $\Delta'\subseteq\Delta$ of the marked diagram $D_I$ is a maximal (by inclusion) connected subdiagram consisting of marked nodes.

Clearly, there is a $1:1$ correspondence between marked diagrams $D_I$ and standard Levi subgroup $L_I$. Thus, there is a $1:1$ correspondence between marked diagrams $D_I$ and standard parabolic subgroups $P_I$. In the context of unipotent conjugacy classes, distinguished unipotent conjugacy classes are in bijection with distinguished parabolic subgroups. The marked diagram corresponding to a distinguished parabolic subgroup is obtained from the labeled Dynkin diagram \cite[\S5.9]{Carter_Finite} by marking those nodes which are labeled by `2'. For example, if $G$ is of type $C_r$, the distinguished unipotent conjugacy classes are in bijection with marked diagrams of the form
$$
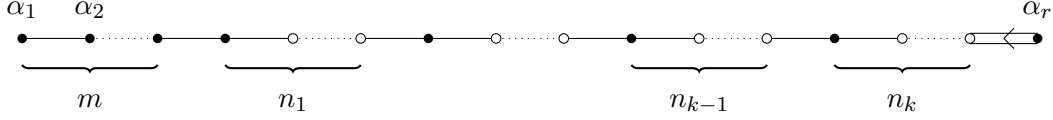
\begin{tikzpicture}[scale=.9]
  \node[label=$\alpha_1$] (N0) at (0,0) {};
  \node[label=$\alpha_2$] (N1) at (1,0) {};
  \node (N2) at (2,0) {};
  \node (N3) at (3,0) {};
  \draw (N0.center) -- (N1.center);
  \fill (N0) circle (2pt);
  \fill (N1) circle (2pt);
  \draw[dotted] (N1.center) -- (N2.center);
  \fill (N2) circle (2pt);
  \draw (N2.center) -- (N3.center);
  \draw [thick,decoration={brace,mirror,raise=10pt},decorate] (N0.center) -- (N2.center) node [pos=0.5,anchor=north,yshift=-17pt] {$m$};

  \node (N4) at (4,0) {};
  \node (N5) at (5,0) {};
  \node (N6) at (6,0) {};
  \draw (N3.center) -- (N4.center);
  \fill (N3) circle (2pt);
  \draw[dotted] (N4.center) -- (N5.center);
  \draw (N5.center) -- (N6.center);
  \fill[white] (N4) circle (2pt);
  \draw (N4) circle (2pt);
  \fill[white] (N5) circle (2pt);
  \draw (N5) circle (2pt);
  \draw [thick,decoration={brace,mirror,raise=10pt},decorate] (N3.center) -- (N5.center) node [pos=0.5,anchor=north,yshift=-17pt] {$n_1$};

  \node (N7) at (7,0) {};
  \node (N8) at (8,0) {};
  \node (N9) at (9,0) {};
  \draw (N6.center) -- (N7.center);
  \draw[dotted] (N7.center) -- (N8.center);
  \draw (N8.center) -- (N9.center);
  \fill (N6) circle (2pt);
  \fill[white] (N7) circle (2pt);
  \draw (N7) circle (2pt);
  \fill[white] (N8) circle (2pt);
  \draw (N8) circle (2pt);

  \node (N10) at (10,0) {};
  \node (N11) at (11,0) {};
  \node (N12) at (12,0) {};
  \draw (N9.center) -- (N10.center);
  \draw[dotted] (N10.center) -- (N11.center);
  \draw (N11.center) -- (N12.center);
  \fill (N9) circle (2pt);
  \fill[white] (N10) circle (2pt);
  \draw (N10) circle (2pt);
  \fill[white] (N11) circle (2pt);
  \draw (N11) circle (2pt);
  \draw [thick,decoration={brace,mirror,raise=10pt},decorate] (N9.center) -- (N11.center) node [pos=0.5,anchor=north,yshift=-17pt] {$n_{k-1}$};

  \node (N13) at (13,0) {};
  \node (N14) at (14,0) {};
  \node[label=$\alpha_r$] (N15) at (15,0) {};
  \draw (N12.center) -- (N13.center);
  \draw[dotted] (N13.center) -- (N14.center);
  \draw[double distance=3.5pt] (N14.center) -- (N15.center);
  \node(N20) at (14.5,0) {};
  \draw (N20.center) -- (N20.north east);
  \draw (N20.center) -- (N20.south east);
  \fill (N12) circle (2pt);
  \fill[white] (N13) circle (2pt);
  \draw (N13) circle (2pt);
  \fill[white] (N14) circle (2pt);
  \draw (N14) circle (2pt);
  \fill (N15) circle (2pt);
  \draw [thick,decoration={brace,mirror,raise=10pt},decorate] (N12.center) -- (N14.center) node [pos=0.5,anchor=north,yshift=-17pt] {$n_k$};


\end{tikzpicture}
$$
where $m+n_1+\dots+n_k+1=r$, $n_1=2$ and where $n_{i+1}=n_i$ or $n_i+1$ for each $1\leq i \leq k-1$. 

In what follows, marked diagrams will be used both in the description of distinguished parabolic subgroups and in the description of standard Levi subgroups of $G$.

A \emph{shift} of a diagram is the transformation under which we obtain a diagram of the same type with one component moved one node to the left or to the right without touching another component. A \emph{permutation} of a diagram is the transformation under which we obtain a diagram of the same type by permuting the components. Under such transformations the corresponding standard Levi subgroups are conjugate (see Lemma \ref{lem_perm_diag}).




\section{Covering numbers of distinguished unipotent conjugacy classes}
\label{sec_distinguished}

Recall that a unipotent element is distinguished if $C_G(u)^{\circ}$ is unipotent. A parabolic subgroup $P=LQ$ with Levi factor $L$ and unipotent radical $Q$ is  distinguished if $\dim(L)=\dim (Q/[Q,Q])$ \cite[\S2.5-6]{Liebeck_Seitz}. An element $g$ of a parabolic subgroup $P$ is called a Richardson element of $P$ if the $P$-conjugacy class of $g$ intersects the unipotent radical $Q$ in an open set of $Q$. By the Bala-Carter-Pommerening classification of unipotent conjugacy classes \cite{Bala_Carter,Pommerening} (see also \cite[Theorem 5.9.6 and \S5.11]{Carter_Finite}), there is a bijection between distinguished conjugacy classes and conjugacy classes of distinguished parabolic subgroups. Under this bijection, the conjugacy class of the distinguished parabolic subgroup $P$ corresponds to the (unique) $G$-conjugacy class containing a Richardson element of $P$.

\begin{lem}
  \label{p_radical}
  Let $P$ be a parabolic subgroup of $G$ with unipotent radical $Q$ and let $N$ be a normal subset of $G$.
  If $N$ contains a Richardson element of $P$ then $\cn(Q,N)\leq 2$.
\end{lem}
\begin{proof}
  Let $C$ be the conjugacy class in $N$ containing a Richardson element of $P$. Since $C$ contains a Richardson element of $P$ it intersects $Q$ in an open subset $V$. Since $V$ is an open subset of the connected group $Q$ we have $Q=V^2\subseteq C^2\subseteq N^2$.
  \end{proof}

The following lemma is known. We give two possible proofs.

\begin{lem}
  \label{regular_ss}
  If $N$ is a normal subset containing a regular semisimple element then $\cn(G,N)\leq 3$.
\end{lem}

\begin{proof}[Proof 1]
  Let $s\in N$ be a regular semisimple element. We may assume $s\in T$. All elements in $sU$ and all elements in $sU^{-}$ are conjugate to $s$ (see \cite[\S2.4]{Humphreys_conjugacy_classes}). Hence $sU,sU^{-}\subseteq N$. Let $(sU)^{U^{-}}$ denote the set of conjugates of elements in $sU$ by elements in $U^{-}$. By Theorem \cite[Theorem 1]{VSS2011unitriangular} we have
  $$
  G=s^3\cdot U\cdot U^{-}\cdot U\cdot U^{-}=(sU)^{U^{-}}\cdot (sU)^{U^{-}}\cdot sU^{-}\subseteq N^3.
  $$
 since $(sU)^{U^{-}}\cdot sU^{-}=\cup_{v\in U^{-}}v^{-1}sUvsU^{-}=\cup_{v\in U^{-}}v^{-1}sUsU^{-}=U^{-}sUsU^{-}$.
\end{proof}

\begin{proof}[Proof 2]
  Let $s\in N$ be a regular semisimple element. We may assume $s\in T$. By \cite[Theorem 2.1]{Gordeev2000} any non-central element of $G$ is conjugate to $vs^2u$ for some $v\in U^{-}$ and some $u\in U$. Since $s$ is regular, $vs$ and $su$ are conjugate to $s$. Hence, $N^2$ contains any non-central element of $G$. Thus $N^3=G$.
  \end{proof}

\begin{lem}
  \label{A2d_roots}
  Let $\Phi$ be of type $A_{2d+1}$ and let $\gamma_k:=\alpha_{d+1}+\sum_{i=1}^{k}\alpha_{d+1+i}+\alpha_{d+1-i}$ for $0\leq k\leq d$.
  The set of roots $R=\{\gamma_k:0\leq k\leq d\}$ has the property that $\alpha+\beta$ is not a root for any $\alpha,\beta\in R\cup -R$.
  \end{lem}

\begin{proof}
  For $\alpha,\beta\in R$ the coefficient of $\alpha_{d+1}$ in a decomposition of $\alpha+\beta$ w.r.t. $\Delta$ is $2$, hence $\alpha+\beta$ is not a root. Similarly for $\alpha,\beta\in -R$. It suffices to notice that $\alpha-\beta$ is not a root for any $\alpha,\beta\in R$. Let $\alpha=\gamma_k$ and $\beta=\gamma_m$ for some $0\leq k,m\leq d$. If $k=m$ then $\alpha-\beta=0$ which is not a root. If $k>m$ then $\alpha-\beta=\delta_1+\delta_2$ with $\delta_1=\sum_{i=m}^k\alpha_{d+1+i}$ and $\delta_2=\sum_{i=m}^k\alpha_{d+1-i}$. Since $\Phi$ is of type $A$ it is easy to see that the two roots are orthogonal, and hence, that their sum is not a root. The case $k<m$ is similar.
  \end{proof}

\begin{lem}
  \label{A2d_rootgroups}
  Let $G$ be of type $A_{2d+1}$, let $I=\Delta-\{\alpha_{d+1}\}$ and let $P_I$ be the corresponding standard parabolic subgroup with unipotent radical $Q_I$. If $N$ is a normal subset of $G$ containing $Q_I$, then $\cn(G,N)\leq 6$. 
\end{lem}

\begin{proof}
  Let $R=\{\gamma_0,\dots,\gamma_d\}$ be the set of roots described in Lemma \ref{A2d_roots}. Since $N$ contains $Q_I$ it also contains $\prod_{k=0}^{d}U_{\gamma_k}$. Let $w_0$ be the longest element (with respect to $\Delta$) of the Weyl group $N_G(T)/T$. One checks that $w_0(\gamma_k)=-\gamma_k$ for all $0\leq k\leq d$. Thus $(\prod_{k=0}^{d}U_{\gamma_k})^{\dot w_0}=\prod_{k=0}^{d}U_{-\gamma_k}$. Since $N$ is a normal subset of $G$ it contains the product of commuting root subgroups $\prod_{i=0}^{d}U_{-\gamma_i}$.

  The factors in the product $\prod_{k=0}^{d}U_{\gamma_k}$ commute since $\gamma_i+\gamma_j$ is not a root for all $0\leq i,j\leq d$. Moreover, since $\alpha+\beta$ is not a root for any $\alpha,\beta\in R\cup -R$ by the commutator relations we have $[U_{\gamma_i},U_{-\gamma_j}]=1$ for $0\leq i\neq j\leq d$. By \cite[Theorem 2.1]{Gordeev2000} any non-central element of $G(\gamma_i)$ is conjugate to an element in $U_{\gamma_i}U_{-\gamma_i}$, hence 
  $$
  \prod_{i=1}^{m} U_{\gamma_i}U_{-\gamma_i}
  =
  \left(\prod_{i=1}^{m} U_{\gamma_i}\right)
  \left(\prod_{i=1}^{m} U_{-\gamma_i}\right)
  $$
  is an open subset of $G(R)$ contained in $N^2$. In particular, an open subset $\tilde T$ of the torus $\prod_{i=0}^{d}T_{\gamma_i}$ lies in $N^2$ where $T_{\gamma_i}$ is the image of the cocharacter $\gamma_i^{\vee}$. A direct check using \eqref{cochar_action} shows that $\tilde T$ does not commute with any root subgroup, hence $C_G(\tilde T)^{\circ}=T$ \cite[II Theorem 4.1]{Springer_Steinberg}. Thus it contains an element $t$ \cite[Lemma 6.4.3]{Springer_algebraic_groups} with $C_G(t)^{\circ}=T$, i.e. it contains a regular semisimple element and the claim follows from Lemma \ref{regular_ss}.
  \end{proof}

\begin{prop}
  \label{classical_G}
  Let $G$ be a classical simple algebraic group of rank $\rk(G)>11$. If $C$ is a distinguished unipotent conjugacy class of $G$ then there exists a torus $\tilde T$ of dimension $\rk(G)$ such that $\cn(\tilde T,C)\leq 36$.
\end{prop}

\begin{proof}
  We prove the statement by means of a case-by-case analysis. Let $P_I$ be the distinguished parabolic subgroup of $G$ with unipotent radical $Q_I$ such that $C\cap Q_I$ is open in $Q_I$. The possible sets of roots $I$ can be read off from the possible distinguished diagrams \cite[\S5.9]{Carter_Finite}. More precisely, $I$ consists of the simple roots corresponding to the nodes labeled with `$0$' in the distinguished diagram corresponding to $P_I$.

  First we treat the case of regular unipotent elements. If $C$ is the conjugacy class of regular unipotent elements then $I=\emptyset$ and $Q_I=U$. Hence $C$ contains an open subset of $U$ and $U^{-}$. Thus $C^2$ contains an open subset of $UU^{-}$ and therefore also an open subset of $(UU^{-})^G$. By \cite[Theorem 2.1]{Gordeev2000} we have $(UU^{-})^G=G-Z(G)$, thus $C^2$ contains an open subset of $G$ and so $C^4=G$.

  
  If $G$ is of type $C_r$ then the distinguished diagrams are 
$$
\begin{tikzpicture}[scale=.9]
  \node[label=$\alpha_1$] (N0) at (0,0) {};
  \node[label=$\alpha_2$] (N1) at (1,0) {};
  \node (N2) at (2,0) {};
  \node (N3) at (3,0) {};
  \draw (N0.center) -- (N1.center);
  \fill (N0) circle (2pt);
  \fill (N1) circle (2pt);
  \draw[dotted] (N1.center) -- (N2.center);
  \fill (N2) circle (2pt);
  \draw (N2.center) -- (N3.center);
  \draw [thick,decoration={brace,mirror,raise=10pt},decorate] (N0.center) -- (N2.center) node [pos=0.5,anchor=north,yshift=-17pt] {$m$};

  \node (N4) at (4,0) {};
  \node (N5) at (5,0) {};
  \node (N6) at (6,0) {};
  \draw (N3.center) -- (N4.center);
  \fill (N3) circle (2pt);
  \draw[dotted] (N4.center) -- (N5.center);
  \draw (N5.center) -- (N6.center);
  \fill[white] (N4) circle (2pt);
  \draw (N4) circle (2pt);
  \fill[white] (N5) circle (2pt);
  \draw (N5) circle (2pt);
  \draw [thick,decoration={brace,mirror,raise=10pt},decorate] (N3.center) -- (N5.center) node [pos=0.5,anchor=north,yshift=-17pt] {$n_1$};

  \node (N7) at (7,0) {};
  \node (N8) at (8,0) {};
  \node (N9) at (9,0) {};
  \draw (N6.center) -- (N7.center);
  \draw[dotted] (N7.center) -- (N8.center);
  \draw (N8.center) -- (N9.center);
  \fill (N6) circle (2pt);
  \fill[white] (N7) circle (2pt);
  \draw (N7) circle (2pt);
  \fill[white] (N8) circle (2pt);
  \draw (N8) circle (2pt);

  \node (N10) at (10,0) {};
  \node (N11) at (11,0) {};
  \node (N12) at (12,0) {};
  \draw (N9.center) -- (N10.center);
  \draw[dotted] (N10.center) -- (N11.center);
  \draw (N11.center) -- (N12.center);
  \fill (N9) circle (2pt);
  \fill[white] (N10) circle (2pt);
  \draw (N10) circle (2pt);
  \fill[white] (N11) circle (2pt);
  \draw (N11) circle (2pt);
  \draw [thick,decoration={brace,mirror,raise=10pt},decorate] (N9.center) -- (N11.center) node [pos=0.5,anchor=north,yshift=-17pt] {$n_{k-1}$};

  \node (N13) at (13,0) {};
  \node (N14) at (14,0) {};
  \node[label=$\alpha_r$] (N15) at (15,0) {};
  \draw (N12.center) -- (N13.center);
  \draw[dotted] (N13.center) -- (N14.center);
  \draw[double distance=3.5pt] (N14.center) -- (N15.center);
  \fill (N12) circle (2pt);
  \fill[white] (N13) circle (2pt);
  \draw (N13) circle (2pt);
  \fill[white] (N14) circle (2pt);
  \draw (N14) circle (2pt);
  \fill (N15) circle (2pt);
  \draw [thick,decoration={brace,mirror,raise=10pt},decorate] (N12.center) -- (N14.center) node [pos=0.5,anchor=north,yshift=-17pt] {$n_k$};

  \node(N16) at (14.5,0) {};
  \draw (N16.center) -- (N16.north east);
  \draw (N16.center) -- (N16.south east);

\end{tikzpicture}
$$
where $m+n_1+\dots+n_k+1=r$, $n_1=2$ and $n_{i+1}=n_i$ or $n_i+1$ for each $1\leq i \leq k-1$.

If $G$ is of type $B_r$ then the distinguished diagrams are 
$$
\begin{tikzpicture}[scale=.9]
  \node[label=$\alpha_1$] (N0) at (0,0) {};
  \node[label=$\alpha_2$] (N1) at (1,0) {};
  \node (N2) at (2,0) {};
  \node (N3) at (3,0) {};
  \draw (N0.center) -- (N1.center);
  \fill (N0) circle (2pt);
  \fill (N1) circle (2pt);
  \draw[dotted] (N1.center) -- (N2.center);
  \fill (N2) circle (2pt);
  \draw (N2.center) -- (N3.center);
  \draw [thick,decoration={brace,mirror,raise=10pt},decorate] (N0.center) -- (N2.center) node [pos=0.5,anchor=north,yshift=-17pt] {$m$};

  \node (N4) at (4,0) {};
  \node (N5) at (5,0) {};
  \node (N6) at (6,0) {};
  \draw (N3.center) -- (N4.center);
  \fill (N3) circle (2pt);
  \draw[dotted] (N4.center) -- (N5.center);
  \draw (N5.center) -- (N6.center);
  \fill[white] (N4) circle (2pt);
  \draw (N4) circle (2pt);
  \fill[white] (N5) circle (2pt);
  \draw (N5) circle (2pt);
  \draw [thick,decoration={brace,mirror,raise=10pt},decorate] (N3.center) -- (N5.center) node [pos=0.5,anchor=north,yshift=-17pt] {$n_1$};

  \node (N7) at (7,0) {};
  \node (N8) at (8,0) {};
  \node (N9) at (9,0) {};
  \draw (N6.center) -- (N7.center);
  \draw[dotted] (N7.center) -- (N8.center);
  \draw (N8.center) -- (N9.center);
  \fill (N6) circle (2pt);
  \fill[white] (N7) circle (2pt);
  \draw (N7) circle (2pt);
  \fill[white] (N8) circle (2pt);
  \draw (N8) circle (2pt);

  \node (N10) at (10,0) {};
  \node (N11) at (11,0) {};
  \node (N12) at (12,0) {};
  \draw (N9.center) -- (N10.center);
  \draw[dotted] (N10.center) -- (N11.center);
  \draw (N11.center) -- (N12.center);
  \fill (N9) circle (2pt);
  \fill[white] (N10) circle (2pt);
  \draw (N10) circle (2pt);
  \fill[white] (N11) circle (2pt);
  \draw (N11) circle (2pt);
  \draw [thick,decoration={brace,mirror,raise=10pt},decorate] (N9.center) -- (N11.center) node [pos=0.5,anchor=north,yshift=-17pt] {$n_{k-1}$};

  \node (N13) at (13,0) {};
  \node (N14) at (14,0) {};
  \node[label=$\alpha_r$] (N15) at (15,0) {};
  \draw (N12.center) -- (N13.center);
  \draw[dotted] (N13.center) -- (N14.center);
  \draw[double distance=3.5pt] (N14.center) -- (N15.center);
  \fill (N12) circle (2pt);
  \fill[white] (N13) circle (2pt);
  \draw (N13) circle (2pt);
  \fill[white] (N14) circle (2pt);
  \draw (N14) circle (2pt);
  \fill[white] (N15) circle (2pt);
  \draw (N15) circle (2pt);
  \draw [thick,decoration={brace,mirror,raise=10pt},decorate] (N12.center) -- (N14.center) node [pos=0.5,anchor=north,yshift=-17pt] {$n_k$};

  \node(N16) at (14.5,0) {};
  \draw (N16.north) -- (N16.east);
  \draw (N16.south) -- (N16.east);

\end{tikzpicture}
$$
where $m+n_1+\dots+n_k+1=r$, $n_1=2$, $n_{i+1}=n_i$ or $n_i+1$ for each $1 \leq i\leq k-2$ and $n_k=n_{k-1}/2$ if $n_{k-1}$ is even or $n_k=(n_{k-1}-1)/2$ if $n_{k-1}$ is odd.

If $G$ is of type $D_r$ then the distinguished diagrams are 
$$
\begin{tikzpicture}[scale=.9]
  \node[label=$\alpha_1$] (N0) at (0,0) {};
  \node[label=$\alpha_2$] (N1) at (1,0) {};
  \node (N2) at (2,0) {};
  \node (N3) at (3,0) {};
  \draw (N0.center) -- (N1.center);
  \fill (N0) circle (2pt);
  \fill (N1) circle (2pt);
  \draw[dotted] (N1.center) -- (N2.center);
  \fill (N2) circle (2pt);
  \draw (N2.center) -- (N3.center);
  \draw [thick,decoration={brace,mirror,raise=10pt},decorate] (N0.center) -- (N2.center) node [pos=0.5,anchor=north,yshift=-17pt] {$m$};

  \node (N4) at (4,0) {};
  \node (N5) at (5,0) {};
  \node (N6) at (6,0) {};
  \draw (N3.center) -- (N4.center);
  \fill (N3) circle (2pt);
  \draw (N4.center) -- (N5.center);
  \draw (N5.center) -- (N6.center);
  \fill[white] (N4) circle (2pt);
  \draw (N4) circle (2pt);
  \fill (N5) circle (2pt);

  \node (N7) at (7,0) {};
  \node (N8) at (8,0) {};
  \node (N12) at (9,0) {};
  \draw (N6.center) -- (N7.center);
  \draw (N7.center) -- (N8.center);
  \draw[dotted] (N8.center) -- (N12.center);
  \fill[white] (N6) circle (2pt);
  \draw (N6) circle (2pt);
  \fill (N7) circle (2pt);
  \fill[white] (N8) circle (2pt);
  \draw (N8) circle (2pt);


  \node (N13) at (10,0) {};
  \node (N14) at (11,0) {};
  \node[label={[shift={(.6,-.4)}]$\alpha_{r-1}$}] (N15) at (12,.5) {};
  \node[label={[shift={(.4,-.4)}]$\alpha_{r}$}] (N16) at (12,-.5) {};
  \draw (N12.center) -- (N13.center);
  \draw (N13.center) -- (N14.center);
  \draw (N14.center) -- (N15.center);
  \draw (N14.center) -- (N16.center);
  \fill[white] (N12) circle (2pt);
  \draw (N12) circle (2pt);
  \fill (N13) circle (2pt);
  \fill[white] (N14) circle (2pt);
  \draw (N14) circle (2pt);
  \fill (N15) circle (2pt);
  \fill (N16) circle (2pt);

  \draw [thick,decoration={brace,mirror,raise=10pt},decorate] (N3.center) -- (N14.center) node [pos=0.5,anchor=north,yshift=-17pt] {$2k$};

  
\end{tikzpicture}
$$
where $m+2k+2=r$, together with
$$
\begin{tikzpicture}[scale=.9]
  \node[label=$\alpha_1$] (N0) at (0,0) {};
  \node[label=$\alpha_2$] (N1) at (1,0) {};
  \node (N2) at (2,0) {};
  \node (N3) at (3,0) {};
  \draw (N0.center) -- (N1.center);
  \fill (N0) circle (2pt);
  \fill (N1) circle (2pt);
  \draw[dotted] (N1.center) -- (N2.center);
  \fill (N2) circle (2pt);
  \draw (N2.center) -- (N3.center);
  \draw [thick,decoration={brace,mirror,raise=10pt},decorate] (N0.center) -- (N2.center) node [pos=0.5,anchor=north,yshift=-17pt] {$m$};

  \node (N4) at (4,0) {};
  \node (N5) at (5,0) {};
  \node (N6) at (6,0) {};
  \draw (N3.center) -- (N4.center);
  \fill (N3) circle (2pt);
  \draw[dotted] (N4.center) -- (N5.center);
  \draw (N5.center) -- (N6.center);
  \fill[white] (N4) circle (2pt);
  \draw (N4) circle (2pt);
  \fill[white] (N5) circle (2pt);
  \draw (N5) circle (2pt);
  \draw [thick,decoration={brace,mirror,raise=10pt},decorate] (N3.center) -- (N5.center) node [pos=0.5,anchor=north,yshift=-17pt] {$n_1$};

  \node (N7) at (7,0) {};
  \node (N8) at (8,0) {};
  \node (N9) at (9,0) {};
  \draw (N6.center) -- (N7.center);
  \draw[dotted] (N7.center) -- (N8.center);
  \draw (N8.center) -- (N9.center);
  \fill (N6) circle (2pt);
  \fill[white] (N7) circle (2pt);
  \draw (N7) circle (2pt);
  \fill[white] (N8) circle (2pt);
  \draw (N8) circle (2pt);

  \node (N10) at (10,0) {};
  \node (N11) at (11,0) {};
  \node (N12) at (12,0) {};
  \draw (N9.center) -- (N10.center);
  \draw[dotted] (N10.center) -- (N11.center);
  \draw (N11.center) -- (N12.center);
  \fill (N9) circle (2pt);
  \fill[white] (N10) circle (2pt);
  \draw (N10) circle (2pt);
  \fill[white] (N11) circle (2pt);
  \draw (N11) circle (2pt);
  \draw [thick,decoration={brace,mirror,raise=10pt},decorate] (N9.center) -- (N11.center) node [pos=0.5,anchor=north,yshift=-17pt] {$n_{k-1}$};

  \node (N13) at (13,0) {};
  \node (N14) at (14,0) {};
  \node[label={[shift={(.6,-.4)}]$\alpha_{r-1}$}] (N15) at (15,.5) {};
  \node[label={[shift={(.4,-.4)}]$\alpha_{r}$}] (N16) at (15,-.5) {};
  \draw (N12.center) -- (N13.center);
  \draw[dotted] (N13.center) -- (N14.center);
  \draw (N14.center) -- (N15.center);
  \draw (N14.center) -- (N16.center);
  \fill (N12) circle (2pt);
  \fill[white] (N13) circle (2pt);
  \draw (N13) circle (2pt);
  \fill[white] (N14) circle (2pt);
  \draw (N14) circle (2pt);
  \fill[white] (N15) circle (2pt);
  \draw (N15) circle (2pt);
  \fill[white] (N16) circle (2pt);
  \draw (N16) circle (2pt);

    \node (N12p) at (12,-0.5) {};
  \draw [thick,decoration={brace,mirror,raise=10pt},decorate] (N12p.center) -- (N16.center) node [pos=0.5,anchor=north,yshift=-17pt] {$n_k$};

  \node (N10p) at (10,0.5) {};

\end{tikzpicture}
$$
where $m+n_1+\dots+n_k=r$, $n_1=2$, $n_{i+1}=n_i$ or $n_i+1$ for each $1 \leq i\leq k-2$ and $n_k=n_{k-1}/2$ if $n_{k-1}$ is even or $n_k=(n_{k-1}+1)/2$ if $n_{k-1}$ is odd.

Choose $t$ maximal such that $t\leq r/2$ and such that $\alpha_{t}$ is a marked node. Let $\Phi_t$ be the root subsystem generated by $I_t=\{\alpha_1,\dots,\alpha_{2t-1}\}$ and let $\Phi_t^{w}$ be the set of non-marked nodes among $\alpha_1,\dots,\alpha_{2t-1}$. Consider the subsystem subgroup $G_t=G(\Phi_t)$. It is a group of type $A_{2t-1}$. By Lemma \ref{p_radical}, $C^2$ contains $Q_I$. In particular it contains $Q_t=Q_I\cap G_t$. The subgroup $Q_t$ is the unipotent radical of the standard parabolic subgroup of $G_t$ with standard Levi factor $L_t$ generated by the maximal torus $T_t=G_t\cap T$ and $G(\Phi_t^{w})$. That is, $Q_t$ is the product (in a fixed but arbitrary order) of the root subgroups $U_{\beta}$ with $\beta\in \Phi^{+}\cap (\Phi_t-\Phi_t^{w})$. Since $\alpha_{t}$ is marked, $L_t$ is a subgroup of the standard parabolic subgroup $P_{I_t-\{\alpha_t\}}$ of $G_t$. Thus, the unipotent radical of this parabolic subgroup is contained in $Q_t\subseteq C^2$. We may therefore apply Lemma \ref{A2d_rootgroups} with $\alpha_{d+1}=\alpha_t$ for the normal subset $C^2\cap G_t$ of the group $G_t$, to obtain that $G_t\subseteq C^{2\cdot 6}$. In particular $C^{12}$ contains the maximal torus $T_t$ of $G_t$.

We claim that $2t-1\geq r/2$. If $t\leq m$ then $\alpha_{t+1}$ is marked and the claim follows. Assume that $t>m$. Let $l$ be such that $m+n_1+\dots+n_l+1=t$ and let $\alpha_{t'}$ be the next marked node to the right of $\alpha_t$. The node $\alpha_{t'}$ exists otherwise $G$ is of type $B_r$ or $D_r$ and $l=k-1$ in which case $n_k\leq \frac{n_{k-1}}{2}<n_{k-1}$. Then $t> r/2$, which is a contradiction with the choice of $t$. If $t'=r$ then $G$ is of type $C_r$ and since $t\leq r/2$, it follows that $k\leq 2$ - a case which is excluded since $r>11$. If $t'=r-1$ then $G$ is of type $D_r$ and since $t\leq r/2$ the rank $r$ would again need to be smaller than $11$. In all other cases $t'=m+n_1+\dots+n_{l+1}+1$. Then, since $n_{l+1}=n_{l}$ or $n_{l+1}=n_{l}+1$, we have
$$
2t-1=2m+2n_1+\dots+2n_l+1\geq m+n_1+\dots+n_l+1=t'
$$
unless $m=0$, $l=1$ and $n_2=n_1+1$ in which case $t=3$ and $t'=6$. In this case, since $t'>r/2$ the rank would again have to be less than $11$.

Thus $\dim T_t\geq r/2$ and there is a Weyl group element $w_1$ such that $T'=T_tT_t^{\dot w_1}\subseteq C^{24}$ is a torus of $G(\alpha_1,\dots,\alpha_{r-1})$ of dimension $r-1$. Indeed, choose $w_1$ to be the longest element of the Weyl group of $G(\alpha_1,\dots,\alpha_{r-1})$.

In all cases, inspecting the root systems one finds that $T'T_t^{\dot w_2}$ contains a torus of dimension $r$, for some element $w_2$ of the Weyl group. There are several such choices and $w_2=w_1s_{\alpha_r}$ works in all cases since $T_t^{\dot w_1}$ projects onto a $1$-dimensional torus of $G(\alpha_{r-1})\cap T$. Thus $C^{36}$ contains an $r$-dimensional torus of $G$.
\end{proof}

\begin{proof}[Proof of Theorem \ref{thm_distinguished}]
  Let $C$ be a distinguished unipotent conjugacy class of $G$. For the bounded rank case we use \cite{Gordeev_Saxl_1}: if $\rk(G)\leq 11$ then, for any conjugacy class $C$ of $G$ we have $\cn(C)\leq 4\cdot\rk(G)\leq 44$. For $\rk(G)>11$, by Proposition \ref{classical_G} there is an $\rk(G)$-dimensional torus in $C^{36}$. Hence $C^{36}$ contains an open subset of $T$. It therefore contains an open subset of $G$ \cite[\S3.5 Corollary]{Steinberg_classes}, hence $G=C^{36\cdot 2}$.
  \end{proof}

\section{Covering numbers of unipotent conjugacy classes in terms of rank}
\label{sec_A}
Let $C$ be the conjugacy class of the unipotent element $u\in G$. By the Bala-Carter-Pommerening classification, up to $G$-conjugacy, there is a unique pair $(L,P)$ consisting of a Levi-subgroup $L$ and a distinguished parabolic subgroup $P$ of $[L,L]$ such that $u$ is a Richardson element of $P$. Conjugating if necessary we may assume that $L=L_I$ and that $P=P_J^{[L_I,L_I]}$ for some $J\subseteq I\subseteq\Delta$. Since it is clear from the context that we consider parabolic subgroups of $[L_I,L_I]$, for brevity we write $P_J$ instead of $P_J^{[L_I,L_I]}$.

\begin{prop}
  \label{rank_Levi_center}
  Let $C$ be the unipotent conjugacy class corresponding to the pair $(L_I,P_J)$.
  We have
  \begin{enumerate}
  \item $\crk(C)=\dim(Z(L_I)),$
  \item $\crk(C)=|\Delta-I|,$
  \item $\rk(C)=|I|.$
    \end{enumerate}
\end{prop}

\begin{proof}
  Let $u\in U$ be a Richardson element of $P_J$ and let $S$ be a maximal torus of $C_G(u)$. The subgroup $L=C_G(S)$ is a Levi subgroup \cite[Proposition 12.10]{Malle_Testerman} which contains $u$. Conjugating if necessary, we may assume that $L$ is the standard Levi subgroup $L_{I'}$ and that $u\in [L_{I'},L_{I'}]$. Let $\tilde S$ be a maximal torus in $C_{[L_{I'},L_{I'}]}(u)$. If $\dim(\tilde S)>0$ then $S\tilde S$ is a torus of $C_G(u)$ of dimension higher than $S$ - a contradiction with the choice of $S$. Hence $\tilde S=\{1\}$ and therefore $C_{[L_{I'},L_{I'}]}(u)^{\circ}$ is unipotent. Thus, $u$ is a distinguished unipotent element of $[L_{I'},L_{I'}]$, i.e. we may assume that $u$ is a Richardson element of some distinguished parabolic subgroup $P_{J'}$ of $[L_{I'},L_{I'}]$. By the Balla-Carter-Pommerening classification the pairs $(L_{I},P_{J})$ and $(L_{I'},P_{J'})$ are $G$-conjugate. Hence $\crk(C)=\dim(Z(L_I))=\dim(S)=\dim(Z(L_{I'}))$ equals $|\Delta-I|$. The last claim also follows since $\rk(C)=\rk(G)-\crk(C)=|\Delta|-|\Delta-I|=|I|$.
\end{proof}
In what follows, marked diagrams will be used to describe standard Levi subgroups. As mentioned in \S\ref{subsection:marked_diagrams}, the components of a marked diagram can be shifted and permuted. Consider a component $D'$ of a marked diagram $D_{\Delta-I}$, i.e. $D'$ is a connected subset of marked nodes corresponding to the roots in $I$. A shift of $D'$ to the left by one position corresponds to removing the marking of the right most node $\gamma'$ of $D'$ and marking $\gamma''$, the first node left of $D'$ (if such a node exists). Let $J=\{\gamma''\}\cup I-\{\gamma'\}$. We say that $D_{\Delta-J}$ is obtained from $D_{\Delta-I}$ by shifting $D'$ to the left with one position if the Levi subgroups corresponding to the two marked diagrams are isomorphic. Similarly one has right shifts. A shift is a left or right shift by any number of positions.

Let $D_1$ and $D_2$ be two components of a marked diagram $D_{\Delta-I}$ which are separated by one non-marked node $\gamma$. We may assume that $D_1$ is left of $\gamma$ and that $D_1$ has more nodes than $D_2$. Let $K_1$ be the set of nodes in $D_1$, let $K_2$ be the set of nodes in $D_2$ and let $K=K_1\cup\{\gamma\}\cup K_2$. Choose $\gamma'\in K$ such that the number of nodes in $K$ left to $\gamma'$ equals $|K_2|$. 
We say that $D_{\Delta-J}$ is obtained from $D_{\Delta-I}$ by permuting the components $D_1$ and $D_2$ if $J=\{\gamma\}\cup I-\{\gamma'\}$ and the Levi subgroups corresponding to the two marked diagrams are isomorphic.

\begin{lem}
  \label{lem_perm_diag}
  Let $L_I$ be a standard Levi subgroup with associated marked diagram $D_I$. If $D_J$ is a marked diagram obtained from $D_I$ by shifting or permuting the components of $D_I$
  then $L_I$ is conjugate to $L_J$.
\end{lem}
\begin{proof}
  Let $\Delta'=\{\alpha_{i+1},\alpha_{i+2},\dots,\alpha_{i+l}\}$ be a component of $D_I$, i.e. $\Delta'$ is the set of roots in $\Delta$ corresponding to a maximal (by inclusion) connected subdiagram of marked nodes in $D_I$:
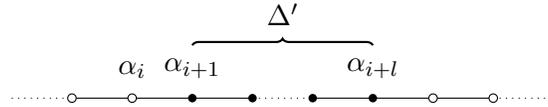
  $$
\begin{tikzpicture}[scale=.8]
  \node (N0) at (0,0) {};
  \node (N1) at (1,0) {};
  \node[label=$\alpha_{i}$] (N2) at (2,0) {};
  \node[label=$\alpha_{i+1}$] (N3) at (3,0) {};
  \node (N4) at (4,0) {};
  \node (N5) at (5,0) {};
  \node[label=$\alpha_{i+l}$] (N6) at (6,0) {};
  \node (N7) at (7,0) {};
  \node (N8) at (8,0) {};
  \node (N9) at (9,0) {};

  \draw[dotted] (N0.center) -- (N1.center);

  \draw (N1.center) -- (N2.center);
  \fill[white] (N1) circle (2pt);
  \draw (N1) circle (2pt);
  
  \draw (N2.center) -- (N3.center);
  \fill[white] (N2) circle (2pt);
  \draw (N2) circle (2pt);
  
  \draw (N3.center) -- (N4.center);
  \fill (N3) circle (2pt);

  \draw[dotted] (N4.center) -- (N5.center);
  \fill (N4) circle (2pt);

  \draw (N5.center) -- (N6.center);
  \fill (N5) circle (2pt);

  \draw (N6.center) -- (N7.center);
  \fill (N6) circle (2pt);
  
  \draw (N7.center) -- (N8.center);
  \fill[white] (N7) circle (2pt);
  \draw (N7) circle (2pt);

  \draw[dotted] (N8.center) -- (N9.center);
  \fill[white] (N8) circle (2pt);
  \draw (N8) circle (2pt);
  
  \draw [thick,decoration={brace,raise=20pt},decorate] (N3.center) -- (N6.center) node [pos=0.5,anchor=north,yshift=40pt] {$\Delta'$};
  
\end{tikzpicture}
$$
Suppose that $\alpha_{i-1}$ and $\alpha_{i}$ are non-marked nodes. If $w_{0}'$ is the longest element (w.r.t. $\Delta$) of the Weyl group of the subsystem subgroup $G(\{\alpha_{i},\dots,\alpha_{i+l}\})$ then $G(\{\alpha_{i+1},\dots,\alpha_{i+l}\})^{\dot w_{0}'}=G(\{\alpha_{i},\dots,\alpha_{i+l-1}\})$. Similarly, if $\alpha_{i+l+1}$ and $\alpha_{i+l+2}$ are not marked, $\Delta'$ can be shifted to the right.

Let $\Delta''=\{\alpha_{i+l+2},\alpha_{i+l+3},\dots,\alpha_{i+l+1+k}\}$ be a second component of $D_I$.
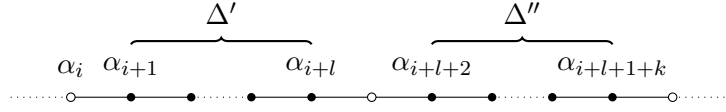
  $$
\begin{tikzpicture}[scale=.8]
  \node (N0) at (0,0) {};
  \node[label=$\alpha_{i}$] (N1) at (1,0) {};
  \node[label=$\alpha_{i+1}$] (N2) at (2,0) {};
  \node (N3) at (3,0) {};
  \node (N4) at (4,0) {};
  \node[label=$\alpha_{i+l}$] (N5) at (5,0) {};
  \node (N6) at (6,0) {};
  \node[label=$\alpha_{i+l+2}$] (N7) at (7,0) {};
  \node (N8) at (8,0) {};
  \node (N9) at (9,0) {};
  \node[label=$\alpha_{i+l+1+k}$] (N10) at (10,0) {};
  \node (N11) at (11,0) {};
  \node (N12) at (12,0) {};

  \draw[dotted] (N0.center) -- (N1.center);

  \draw (N1.center) -- (N2.center);
  \fill[white] (N1) circle (2pt);
  \draw (N1) circle (2pt);
  
  \draw (N2.center) -- (N3.center);
  \fill (N2) circle (2pt);
  
  \draw[dotted] (N3.center) -- (N4.center);
  \fill (N3) circle (2pt);

  \draw (N4.center) -- (N5.center);
  \fill (N4) circle (2pt);

  \draw (N5.center) -- (N6.center);
  \fill (N5) circle (2pt);

  \draw (N6.center) -- (N7.center);
  \fill[white] (N6) circle (2pt);
  \draw (N6) circle (2pt);
  
  \draw (N7.center) -- (N8.center);
  \fill (N7) circle (2pt);

  \draw[dotted] (N8.center) -- (N9.center);
  \fill (N8) circle (2pt);

  \draw (N9.center) -- (N10.center);
  \fill (N9) circle (2pt);

  \draw (N10.center) -- (N11.center);
  \fill (N10) circle (2pt);
  
  \draw[dotted] (N11.center) -- (N12.center);
  \fill[white] (N11) circle (2pt);
  \draw (N11) circle (2pt);

  \draw [thick,decoration={brace,raise=20pt},decorate] (N2.center) -- (N5.center) node [pos=0.5,anchor=north,yshift=40pt] {$\Delta'$};
  
  \draw [thick,decoration={brace,raise=20pt},decorate] (N7.center) -- (N10.center) node [pos=0.5,anchor=north,yshift=40pt] {$\Delta''$};
\end{tikzpicture}
$$
If $w_{0}'$ is the longest element (with respect to $\Delta$) of the Weyl group of the subsystem subgroup $G(\Delta'\cup \{\alpha_{i+l+1}\}\cup\Delta'')$ then $G(\{\alpha_{i+1},\dots,\alpha_{i+l}\})^{\dot w_{0}'}=G(\{\alpha_{i+l+1+k},\dots,\alpha_{i+2+k}\})$ and $G(\{\alpha_{i+l+2},\dots,\alpha_{i+l+1+k}\})^{\dot w_{0}'}=G(\{\alpha_{i+k},\dots,\alpha_{i+1}\})$. This proves the claim for $G$ of type $A_r$.

If $G$ is of type $C_r$ or $B_r$ let $\Delta'$ be the component of $D_I$ containing $\alpha_r$ and if $G$ is of type $D_r$ let $\Delta'$ be the component of $D_I$ containing $\alpha_{r}$ and $\alpha_{r-1}$. If such a component exists, then $\Delta'=\{\alpha_{k_0},\dots \alpha_r\}$ for some $1\leq k_0\leq r$. Since $\Delta'$ is a component of both $D_I$ and $D_J$ and since the other components are shifted and permuted freely in the subsystem subgroup $G(\{\alpha_{1},\dots,\alpha_{k_0-2}\})$ of type $A_{k_0-2}$, the claim follows.
  \end{proof}


\begin{prop}
  \label{rank_classical}
  Let $G$ be a classical simple algebraic group of rank $\rk(G)>11$ defined over an algebraically closed field of good characteristic. If $C$ is a unipotent conjugacy class of $G$ corresponding to the pair $(L_I,P_J)$ then
  $$
  \cn(G,C)\leq c\cdot \frac{\rk(G)}{\rk(C)}=
  c\cdot \left(1+\frac{\crk(C)}{\rk{(C)}}\right).
  $$
  Moreover, we may choose $c\leq 288$.
\end{prop}

\begin{proof}
  Consider the marked diagram $D_{\Delta-I}$ with set of marked nodes $I$. With Lemma \ref{lem_perm_diag}, conjugating, we may shift the components of $D_{\Delta-I}$ to the right of the diagram such that they are separated by exactly one non-marked node. Let $\hat I$ be the minimal subset of simple roots spanning an irreducible root subsystem containing $I$ and $\alpha_r$. Then $L_I$ is included in $L_{\hat I}$. Notice that under this operation of moving all components to the right, it may happen that $\alpha_r$ is non-marked, in which case it is easy to see that $\alpha_{r-1}$ is marked. 
  The sets $I$ and $\hat I$ may be visualised as follows
$$
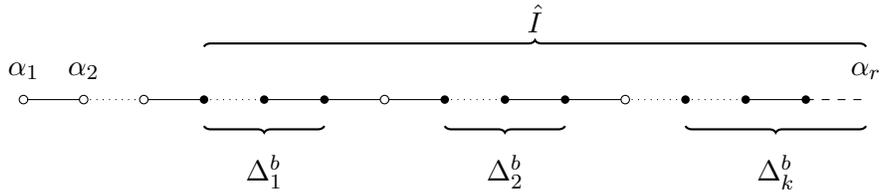
\begin{tikzpicture}[scale=.8]
  \node[label=$\alpha_1$] (N0) at (0,0) {};
  \node[label=$\alpha_2$] (N1) at (1,0) {};
  \node (N2) at (2,0) {};
  \node (N3) at (3,0) {};
  \node (N4) at (4,0) {};
  \node (N5) at (5,0) {};
  \node (N6) at (6,0) {};
  \node (N7) at (7,0) {};
  \node (N8) at (8,0) {};
  \node (N9) at (9,0) {};
  \node (N10) at (10,0) {};
  \node (N11) at (11,0) {};
  \node (N12) at (12,0) {};
  \node (N13) at (13,0) {};
  \node[label=$\alpha_r$] (N14) at (14,0) {};

  \draw (N0.center) -- (N1.center);
  \fill[white] (N0) circle (2pt);
  \draw (N0) circle (2pt);

  \draw[dotted] (N1.center) -- (N2.center);
  \fill[white] (N1) circle (2pt);
  \draw (N1) circle (2pt);

  \draw (N2.center) -- (N3.center);
  \fill[white] (N2) circle (2pt);
  \draw (N2) circle (2pt);
  
  \draw[dotted] (N3.center) -- (N4.center);
  \fill (N3) circle (2pt);

  \draw (N4.center) -- (N5.center);
  \fill (N4) circle (2pt);

  \draw (N5.center) -- (N6.center);
  \fill (N5) circle (2pt);

  \draw (N6.center) -- (N7.center);
  \fill[white] (N6) circle (2pt);
  \draw (N6) circle (2pt);
  
  \draw[dotted] (N7.center) -- (N8.center);
  \fill (N7) circle (2pt);

  \draw (N8.center) -- (N9.center);
  \fill (N8) circle (2pt);

  \draw (N9.center) -- (N10.center);
  \fill (N9) circle (2pt);

  \draw[dotted] (N10.center) -- (N11.center);
  \fill[white] (N10) circle (2pt);
  \draw (N10) circle (2pt);
  
  \draw[dotted] (N11.center) -- (N12.center);
  \fill (N11) circle (2pt);

  \draw (N12.center) -- (N13.center);
  \fill (N12) circle (2pt);

  \draw[dashed] (N13.center) -- (N14.center);
  \fill (N13) circle (2pt);




  \draw [thick,decoration={brace,mirror,raise=10pt},decorate] (N3.center) -- (N5.center) node [pos=0.5,anchor=north,yshift=-17pt] {$\Delta_1^b$};
  
  \draw [thick,decoration={brace,mirror,raise=10pt},decorate] (N7.center) -- (N9.center) node [pos=0.5,anchor=north,yshift=-17pt] {$\Delta_{2}^{b}$};

  \draw [thick,decoration={brace,mirror,raise=10pt},decorate] (N11.center) -- (N14.center) node [pos=0.5,anchor=north,yshift=-17pt] {$\Delta_{k}^{b}$};
  \draw [thick,decoration={brace,raise=20pt},decorate] (N3.center) -- (N14.center) node [pos=0.5,anchor=north,yshift=40pt] {$\hat I$};

\end{tikzpicture}
$$
where the dashed line is one of the following diagrams
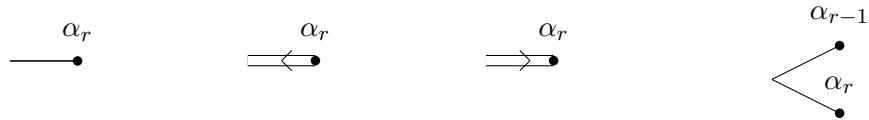
\begin{figure}[h!]
  \centering
  \hfill
  \begin{subfigure}[h]{0.2\textwidth}
    \centering
    \begin{tikzpicture}[scale=.9]
      \node(N0) at (0,0) {};
      \node[label=$\alpha_{r}$] (N1) at (1,0) {};
      \node (N2) at (2,0) {};
      \node (N3) at (3,0) {};
      \draw (N0.center) -- (N1.center);
      \draw (N0.center) -- (N1.center);
      \fill (N1) circle (2pt);
    \end{tikzpicture}
    \label{fig:A}
  \end{subfigure}
  \begin{subfigure}[h]{0.2\textwidth}
    \centering
    \begin{tikzpicture}[scale=.9]
      \node(N0) at (0,0) {};a
      \node[label=$\alpha_r$] (N1) at (1,0) {};
      \node (N2) at (2,0) {};
      \node (N3) at (3,0) {};
      \draw[double distance=3.5pt] (N0.center) -- (N1.center);
      \fill (N1) circle (2pt);
         
      \node(N3) at (0.5,0) {};
      \draw (N3.center) -- (N3.north east);
      \draw (N3.center) -- (N3.south east);
       \end{tikzpicture}
       \label{fig:G2}
  \end{subfigure}
  \begin{subfigure}[h]{0.2\textwidth}
    \centering
    \begin{tikzpicture}[scale=.9]
      \node(N0) at (0,0) {};a
      \node[label=$\alpha_r$] (N1) at (1,0) {};
      \node (N2) at (2,0) {};
      \node (N3) at (3,0) {};
      \draw[double distance=3.5pt] (N0.center) -- (N1.center);
      \fill (N1) circle (2pt);
      
      \node(N3) at (0.5,0) {};
      \draw (N3.north) -- (N3.east);
      \draw (N3.south) -- (N3.east);
    \end{tikzpicture}
    \label{fig:G2}
  \end{subfigure}
  \begin{subfigure}[h]{0.2\textwidth}
    \centering
    \begin{tikzpicture}[scale=.9]
      \node (N0) at (0,0) {};
      \node[label=$\alpha_{r-1}$] (N1) at (1,0.5) {};
      \node[label=$\alpha_{r}$] (N2) at (1,-0.5) {};
      \draw (N0.center) -- (N1.center);
      \draw (N0.center) -- (N2.center);
      \fill (N1) circle (2pt);
      \fill (N2) circle (2pt);
    \end{tikzpicture}
    \label{fig:G2}
  \end{subfigure}
\end{figure}

\par\noindent
with the node $\alpha_r$ possibly non-marked. If $\alpha_r$ is marked, it belongs to the last component $\Delta_k^b$ else $\alpha_{r-1}$ belongs to $\Delta_k^{b}$.

By Proposition \ref{classical_G}, the normal subset $C^{36}$ contains a torus $\tilde T\subseteq [L_I,L_I]\cap T$ of dimension $|I|$. Let $\gamma_1,\dots,\gamma_l$ be the simple roots corresponding to the non-marked nodes $\hat I-I$. Let $w$ be the product of the simple reflections in the roots $\gamma_i$. Then $T'=\tilde T\cdot \tilde T^{\dot w}\subseteq C^{72}$ is a torus of dimension $|\hat I|$ in $T\cap G(\hat I)$ (the Lie algebra of this torus has dimension $|\hat I|$). If only $\alpha_r$ is marked, we may assume that $\hat I=\{\alpha_{r-1},\alpha_r\}$. Let $T''=G(\hat I-\{\alpha_r\})\cap T'$ and notice that $T'\cap G_{\alpha_r}$ is a torus of dimension $1$.

Divide the set of simple roots in $\Delta-\hat I$ into subsets $J_1,J_2,\dots, J_s$ of consecutive roots with $|J_1|\leq |I|$ and $|J_j|=|I|$ for $j\geq 2$. For $1\leq i\leq s$ let $w_i$ be the longest element of the Weyl group generated by the reflections in the roots $J_i\cup J_{i+1}\cup\dots\cup J_s\cup (\hat I-\{\alpha_r\})$. Then
$$
T''(T'')^{\dot w_1}(T'')^{\dot w_2}\cdots (T'')^{\dot w_s}
$$
is a torus in $G(\Delta-\{\alpha_r\})$ of dimension $r-1$ (the Lie algebra of this torus has the right dimension). Thus, since $T'$ contains a $1$-dimensional torus of $G_{\alpha_r}$,
$$
T'(T')^{\dot w_1}(T')^{\dot w_2}\cdots (T')^{\dot w_s}\subseteq C^{72\cdot (s+1)}
$$
is an $r$-dimensional torus of $G$. In particular, $C^{72\cdot (s+1)}$ contains an open subset of $T$ and therefore $C^{144\cdot (s+1)}=G$. Moreover
$$
\cn(G,C)
\leq 144\cdot (s+1)
\leq 144\left(\frac{r}{|I|}+1\right)
\leq 144\cdot 2\frac{r}{|I|}
= 288\frac{\rk(G)}{\rk(C)}
.
$$
\end{proof}

\begin{proof}[Proof of Theorem \ref{thm_A}]
  For the bounded rank case we use \cite{Gordeev_Saxl_1}: if $\rk(G)\leq 11$ then, for any conjugacy class $C$ of $G$ we have $\cn(C)\leq 4\cdot\rk(G)\leq 4\cdot\rk(G)\cdot \frac{\rk(G)}{\rk(C)}\leq 44\cdot \frac{\rk(G)}{\rk(C)}$. Thus we may assume that $G$ is a classical group of rank greater than $11$ and the result follows form Proposition \ref{rank_classical}.
  \end{proof}

\section{Covering numbers of unipotent conjugacy classes in terms of dimension}
\label{sec_B}

In this section we prove Theorem \ref{thm_B}. For the bounded rank case we use \cite{Gordeev_Saxl_1}: if $\rk(G)\leq 8$ then, for any conjugacy class $C$ of $G$ we have $\cn(C)\leq 4\cdot\rk(G)\leq 4\cdot \frac{\dim(G)}{\dim(C)}$. Thus we may assume that $\rk(G)>8$ in which case $G$ is a classical group.

Let $C$ be a unipotent conjugacy class of $G$ corresponding to the pair $(L_I,P_J)$.
Let $D_{\Delta-I}$ be the marked diagram with marked nodes $I$. With Lemma \ref{lem_perm_diag}, conjugating we may shift the components of $D_{\Delta-I}$ to the right of the diagram such that they are separated by exactly one non-marked node. As in the proof of Proposition \ref{rank_classical}, let $\hat I$ be the minimal subset of simple roots spanning an irreducible root subsystem containing $I$ and $\alpha_r$. Then $L_I$ is included in $L_{\hat I}$ and $|\hat I|\leq 2|I|$.
  
Let $u\in L_I$ be a representative of $C$. Let $r=\rk(G)$, $\hat r_b=|\hat I|$ and $\hat r_w=r-\hat r_b$. Since $u\in [L_I,L_I]\subseteq [L_{\hat I},L_{\hat I}]$ there is an $A_{\hat r_w-1}$ subsystem subgroup of $G$ in $C_G(u)$. Hence
$$
\dim(G)-\hat r_w^2+1\geq \dim(C).
$$
Thus, if $G$ is of type $A_r$ then
$$
6r\hat r_b
\geq r^2+2r-\hat r_w^2+1
\geq \dim(C).
$$
For the other classical groups we use a maximal closed subsystem subgroup $H\times [L_{\hat I},L_{\hat I}]$ (recognized with the algorithm of Borel and de Siebenthal, see for example \cite[Theorem 13.12]{Malle_Testerman}). Notice that since $u\in [L_{\hat I},L_{\hat I}]$ we have $H\subseteq C_G(u)$ and so
$$
\dim(G)-\dim(H)\geq \dim(C).
$$
If $G$ is of type $B_r$ then $H$ is of type $D_{\hat r_w}$ and
$$
8r\hat r_b
\geq
2(r-\hat r_w+\frac{1}{2})(r+\hat r_w)
=
2r^2+r-2\hat r_w^2+\hat r_w\geq\dim(C).
$$
If $G$ is of type $C_r$ then $H$ is of type $C_{\hat r_w}$ and
$$
8r\hat r_b
\geq
2(r-\hat r_w)(r+\hat r_w+\frac{1}{2})
=
2r^2+r-2\hat r_w^2-\hat r_w\geq\dim(C).
$$
If $G$ is of type $D_r$ then $H$ is of type $D_{\hat r_w}$ and
$$
4r\hat r_b
\geq
2(r-\hat r_w)(r+\hat r_w)
\geq
2r^2-r-2\hat r_w^2+\hat r_w\geq\dim(C).
$$
In all cases we have
$$
\frac{\rk(G)}{16\rk(C)}=
\frac{r^2}{8(2\rk(C)) r}
\leq
\frac{r^2}{8\hat r_b r}
\leq
\frac{\dim(G)}{\dim(C)}
$$
since $\hat r_b=|\hat I|\leq 2|I|=2\rk(C)$. Thus, by Theorem \ref{thm_A}, $\cn(G,C)\leq 4608\cdot \dim(G)/\dim(C)$.

\section{Conflict of Interest}
The author declares that he has no conflict of interest.

\end{document}